\documentclass{amsart}
\usepackage{epstopdf}
\usepackage{graphicx}
\usepackage{amsfonts}

\usepackage[mathscr]{eucal}

\theoremstyle{plain}
\newtheorem{prop}{Proposition}[section]
\newtheorem{coro}[prop]{Corollary}

\newtheorem{conj}[prop]{Conjecture}
\newtheorem{lemm}[prop]{Lemma}
\newtheorem{ques}[prop]{Question}
\newtheorem{thm}[prop]{Theorem}

\theoremstyle{definition}
\newtheorem{defn}[prop]{Definition}

\newtheorem{rem}[prop]{Remark}

\DeclareMathOperator{\breadth}{breadth}

\def\mcg#1;#2{\Gamma_{#1,#2}}
\def\fg#1;#2{\Pi_{#1,#2}}
\def\tb#1;#2{\mathscr{K}_{\frac{#1}{#2}}}

\begin{document}

\title[A new obstruction of quasi-alternating links]
{A new obstruction of quasi-alternating links}


\author{Khaled Qazaqzeh}
\address{Department of Mathematics\\ Faculty of Science \\ Kuwait University\\
P. O. Box 5969\\ Safat-13060, Kuwait, State of Kuwait}
\email{khaled@sci.kuniv.edu.kw}

\author{Nafaa Chbili}
\address{Department of Mathematical Sciences\\ College of Science UAE University \\
17551 Al Ain, U.A.E.}
\email{nafaachbili@uaeu.ac.ae}
\urladdr{http://faculty.uaeu.ac.ae/nafaachbili}

\date{02/06/2014}

\begin{abstract}
We prove that the degree of the Brandt-Lickorish-Millet polynomial
of any quasi-alternating link is less than its determinant.
Therefore, we obtain a new and a simple obstruction criterion  for
quasi-alternateness. As an application, we identify some knots of 12
crossings or less and some links of 9 crossings  or less that are
not quasi-alternating. 
Also, we show that there are only finitely many Kanenobu knots which
are quasi-alternating. This last result supports Conjecture 3.1 of
Greene in \cite{G2} which states that there are only finitely many
quasi-alternating links with a given determinant. Moreover, we
identify an infinite family of non quasi-alternating Montesinos
links and this supports Conjecture 3.10 in \cite{QCQ} that
characterizes quasi-alternating Montesinos links.
\end{abstract}

\maketitle

\section{introduction}

Quasi-alternating links were first introduced by Ozsv$\acute{a}$th
and Szab$\acute{o}$ in \cite{OS}. This class of links appeared in
the context of link homology as a natural generalization of
alternating links. Quasi-alternating links are defined recursively
as follows:

\begin{defn}\label{def}
The set $\mathcal{Q}$ of quasi-alternating links is the smallest set
satisfying the following properties:
\begin{itemize}
    \item The unknot belongs to $\mathcal{Q}$.
  \item If $L$ is a link with a diagram $D$ containing a crossing $c$ such that
\begin{enumerate}
\item both smoothings of the diagram $D$ at the crossing $c$, $L_{0}$ and $L_{\infty}$
as in figure \ref{figure} belong to $\mathcal{Q}$, and
\item $\det(L_{0}), \det(L_{\infty}) \geq 1$,
\item $\det(L) = \det(L_{0}) + \det(L_{\infty})$; then $L$ is in $\mathcal{Q}$
and in this case we say that  $L$ is quasi-alternating at the crossing $c$
with quasi-alternating diagram $D$.
\end{enumerate}
\end{itemize}
\end{defn}

\begin{figure} [h]
\begin{center}
\includegraphics[width=10cm,height=2cm]{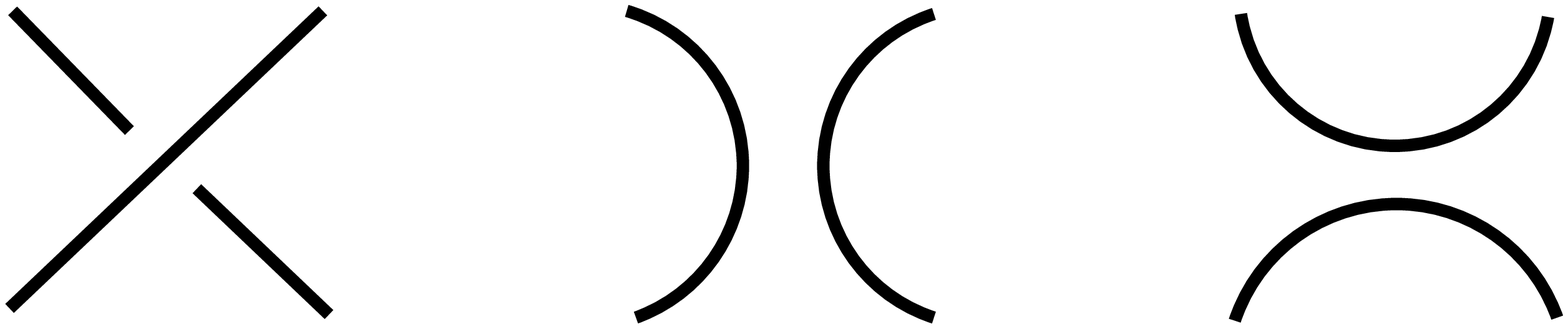} \\
{$L$}\hspace{3.5cm}{$L_0$}\hspace{3.5cm}$L_{\infty}$
\end{center}
\caption{The link diagram $L$ at crossing $c$  and its smoothings $L_{0}$ and $L_{\infty}$
respectively.}\label{figure}
\end{figure}

Here is a list of properties of alternating links that hold for
quasi-alternating links as well. These are actually the main
obstruction criteria that have been used to study
quasi-alternateness of links.\\

\begin{enumerate}
\item the branched double-cover of any quasi-alternating link
is an $L$-space \cite[Proposition\,3.3]{OS};
\item the space of branched double-cover of any
quasi-alternating link bounds a negative definite $4$-manifold $W$
with $H_{1}(W) = 0$ \cite[Proof of Lemma\,3.6]{OS};
\item the $\mathbb{Z}/2\mathbb{Z}$ knot Floer homology group of any
quasi-alternating link is thin \cite[Theorem\,2]{MO};
\item the reduced ordinary Khovanov homology group of any
quasi-alternating link is thin \cite[Theorem\,1]{MO};
 \item the reduced odd Khovanov homology group of any
quasi-alternating link is thin \cite[Remark after
Proposition\,5.2]{ORS}.
\end{enumerate}

The recursive nature of the definition of the set $\mathcal{Q}$
makes it hard to decide whether a given link is quasi-alternating by
the only use of the  definition. A different approach for the study
of  this  class of links is to find   algebraic  properties that
characterize quasi-alternating links. This is actually  the main
motivation of this paper. We prove  that the degree of the
Brandt-Lickorish-Millet  polynomial $Q_L(x)$ of any
quasi-alternating link $L$ is less than its determinant. Hence, we
obtain a new property of quasi-alternating links which is in other
words a new  and simple obstruction criterion for
quasi-alternateness.
\begin{thm}\label{main}
For any quasi-alternating link $L$, we have $\deg Q_{L}  < \det(L)$.
\end{thm}

Consequently, we provide a table of  knots with up to 12 crossings,
and a table of links with up to 9 crossings which are not
quasi-alternating. Also, we show that there are only finitely many
Kanenobu knots that are quasi-alternating. This gives an easier
proof of one of the claims of \cite[Theorem\,2]{GW} and supports
Conjecture 3.1 of Greene in \cite{G2} since all Kanenobu knots have
equal determinant. In addition, we identify an infinite family of
non quasi-alternating Montesinos links. The later result supports
Conjecture 3.10 in \cite{QCQ} that characterizes quasi-alternating
Montesinos links.


\section{Proof and Applications}






In 1984, Jones introduced  a new polynomial  $V_L(t)$ which is an
invariant of ambient isotopy of oriented links in the three-sphere.
The Jones polynomials can be defined recursively by the following
relations:
$$\begin{array}{l}
V_U(t)=1\\
tV_{L_+}(t) - t^{-1} V_{L_-}(t)= (\sqrt{t}+\displaystyle\frac{1}{\sqrt{t}}) V_{L_0}(t),
\end{array}$$
where $U$ is the unknot and  $L_{+}$, $L_{-}$ and  $L_{0}$ are three links
which are identical except in a small ball where they are as pictured below:

\begin{figure}[h]
\centering
\includegraphics[width=10cm,height=2.5cm]{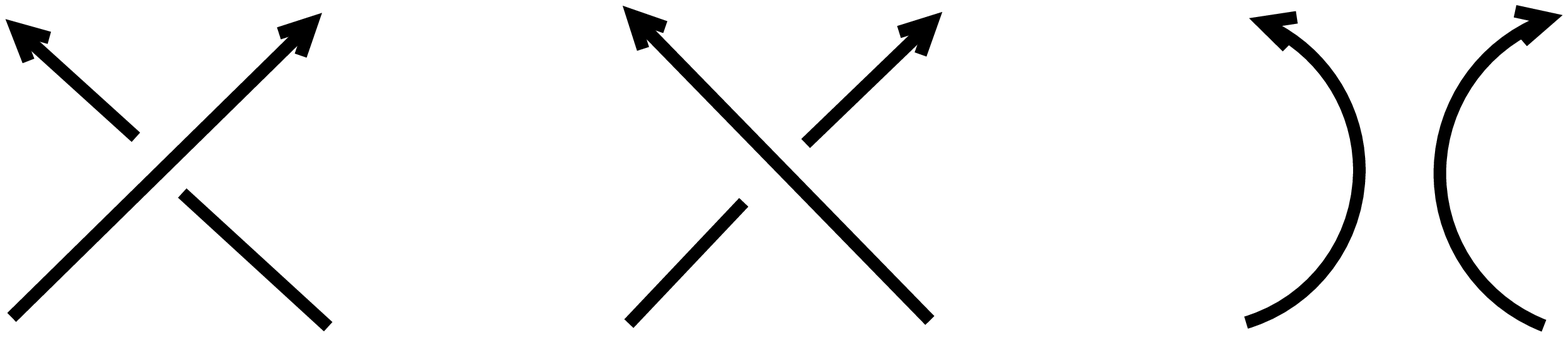}
\caption{ $L_+$, $L_{-}$ and  $L_0$,  respectively.}
\end{figure}

Shortly after the discovery of the Jones polynomial, Brandt,
Lickorish and Millet in \cite{BLM} introduced a new link invariant
$Q_L(x)$. For any link $L$, $Q_L(x)$ is a Laurent polynomial which
can be  defined by $Q_U(x)=1$ and a recursive
relation on link diagrams as follows:\\

$$\begin{array}{l}
Q_{L_+}(x) +  Q_{L_-}(x)= x (Q_{L_0}(x) +  Q_{L_\infty}(x))
\end{array}$$
where  $L_{+}$, $L_{-}$,  $L_{0}$ and $L_\infty$  are four
links which are identical except in a small ball where they are as in the following picture
\begin{figure}[h]
\centering
\includegraphics[width=12cm,height=2.5cm]{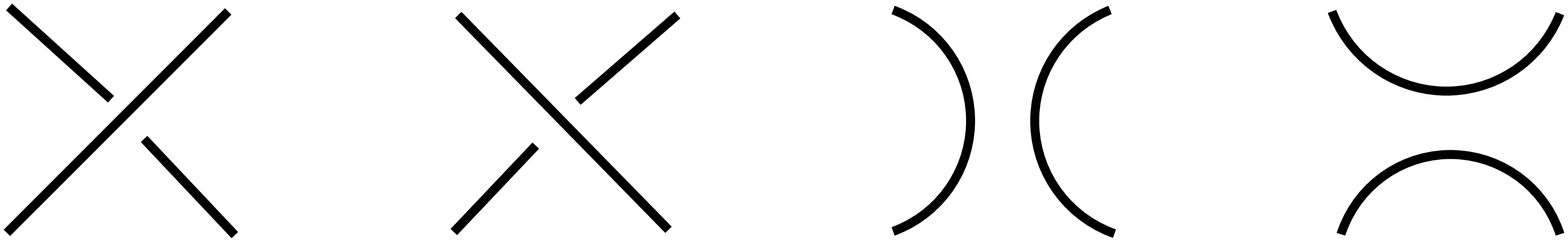}
\caption{$L_+$, $L_-$, $L_0$ and $L_{\infty}$, respectively.}
\label{figure8}
\end{figure}

It is worth mentioning that the  $Q$-polynomial is a specialization of the well known two-variable Kauffman
polynomial $F$, \cite{BLM}. More precisely, for any link $L$, we have $ Q_{L}(x) = F_{L}(1, x)$.
In the remaining part of the paper, $\deg (Q)$ refers to the highest power of $x$ that appears
in $Q(x)$. It is well known that   $\deg (Q)$ is always positive.
Here are some properties of the $Q-$polynomial.\\
\begin{prop}[\cite{BLM}]
The $Q$-polynomial satisfies the following:
\begin{enumerate}
 \item if $L$ is a $k-$component link, then the lowest power  of $x$ that appears in $Q$ is $1-k$.
     \item $Q(L_1 \# L_2)$ = $Q(L_1)$ $Q(L_2)$, where $L_1 \# L_2$ is the connected sum of $L_1$ and $L_2$.
    \item $Q(L')$ = $Q(L)$, where $L'$ is the mirror image of $L$.
\end{enumerate}
\end{prop}

The following lemma is the key step towards  the proof of the main result of this paper.
\begin{lemm}\label{simple}
Let  $L$ be a   link,  then
\[
\deg Q_{L} \leq  \max \{\deg Q_{L_{0}},  \deg Q_{L_{\infty}}\}+1,
\]
where $L_{0}, L_{\infty}$ are the smoothings of the link $L$ at any
crossing $c$.
\end{lemm}
\begin{proof}

For any link diagram $D$, we define $n(D)$ to be  the minimum number of crossing switches
\includegraphics[width=1cm]{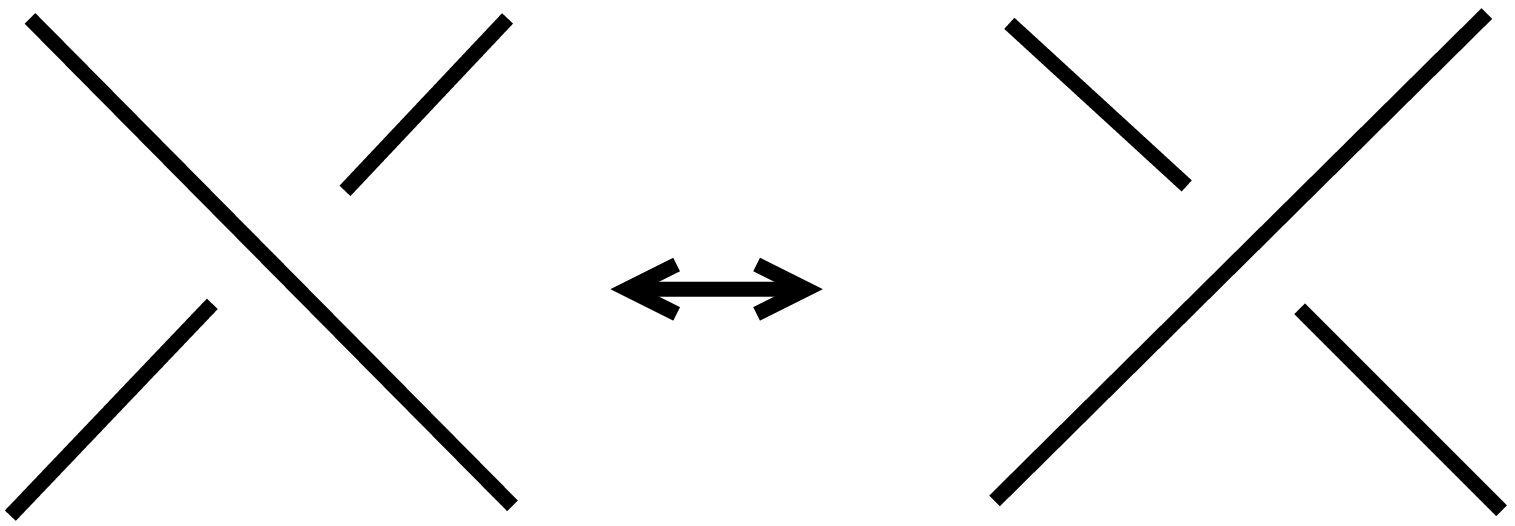} necessary to transform $D$ into a diagram
of the unlink. We prove the lemma by induction on  $n=n(D)$. First,
notice that we can assume without loss of generality that $D=D_+$.
If  $n = 1$, then there is a crossing in $D$ whose change produces a
diagram $D_{-}$ of  the unlink with $k$ components $L_{-}$. Since
$Q_{L_{-}}(x)=(2x^{-1}-1)^{k-1}$, then the skein relation writes: $
Q_{L}(x) = x (Q_{L_0}(x) + Q_{L_\infty}(x))-(2x^{-1}-1)^{k-1}$.
Using the fact that the degree of  $Q_{L}$ is always positive, we
can see that the term $(2x^{-1}-1)^{k-1}$ will not contribute to the
coefficient of the highest degree of $Q_{L}$. Consequently, $\deg
Q_{L} \leq \deg(Q_{L_0}+Q_{L_{\infty}})+1 \leq \max  \{\deg
Q_{L_{0}},  \deg Q_{L_{\infty}}\}+1$.

Now, let $L$ be a link having a diagram $D$  whose number of crossing switches is $n(D)=n$.
Assume  the result is true for all link diagrams  with a crossing switches less
than  $n$, in particular for the link $L_{-}$.  Since
$
Q_{L}(x) = x (Q_{L_0}(x) + Q_{L_\infty}(x))-Q_{L_-}(x)
 $.
Then, $\deg Q_{L}(x) \leq \max \{\deg(Q_{L_0}+Q_{L_{\infty}})+1,
\deg Q_{L_{-}}\} \leq \max \{ \deg Q_{L_0}, \deg Q_{L_{\infty}}
\}+1.$ 

\end{proof}

\begin{proof}[\textbf{Proof of Theorem 1.2}]
We use induction on the determinant of the given quasi-alternating
link $L$. The result is obvious if $\det(L)
= 1$ since  the only link  that
is quasi-alternating with determinant 1 is the unknot. Now assume that the result
is true for all quasi-alternating links with determinant less than
or equal to $m$. If $L$ is a quasi-alternating link with determinant
$m+1$, then both  $\det(L_{0})$ and $\det(L_{\infty})$ are less than or equal  to  $m$. By the induction assumption
 $\deg Q_{L_{0}} <
\det(L_{0})$ and $\deg Q_{L_{\infty}} < \det(L_{\infty})$. Consequently:
$$\begin{array}{ll}
\deg Q_{L} &\leq  \max \{\deg Q_{L_{0}},  \deg Q_{L_{\infty}}\}+1 \\
& <  \max \{\det(L_{0}),  \det{L_{\infty}}\}+1\\
& < \det(L_{0})+\det(L_{\infty})\\
&< \det(L).
\end{array}
$$
\end{proof}

In \cite{K}, Kanenobu  introduced  an infinite family of knots
$K(p,q)$, where  $p,q$ are two integers.  All these knots are known
to have determinant equal to 25. Our idea is to  apply  Theorem
\ref{main}  to study the quasi-alternateness of Kanenobu knots. The
$Q$-polynomial of any Kanenobu knot is given by the following
proposition:

\begin{prop}
[Proposition 4.5, \cite{K}] Let $Q(a,b)$ be the $Q$-polynomial of
the Kanenobu knot $K(a,b)$, then we have
\begin{align*}
Q(a,b) = -\sigma_{a}\sigma_{b}(Q(8_{9})-1) +
x^{-1}(\sigma_{a+1}\sigma_{b+1} +
\sigma_{a-1}\sigma_{b-1})(Q(8_{8})-1) + 1,
\end{align*}
where $Q(8_{8}) = 1 + 4x + 6x^{2} - 10x^{3} - 14x^{4} +4x^{5}
+8x^{6} +2x^{7}$ and $Q(8_{9}) = -7 + 4x + 16x^{2} - 10x^{3} -
16x^{4} +4x^{5} +8x^{6} +2x^{7}$.
\end{prop}

In the proposition above, $\sigma_{n}$ is defined as follows:

\[ \sigma_{n} =
  \begin{cases}
  \frac{\alpha^{n} - \beta^{n}}{\alpha - \beta}, & \text{if}\  n > 0,
     \\
        0, & \text{if} \ n = 0,
     \\
  -\frac{\alpha^{-n} - \beta^{-n}}{\alpha - \beta}, & \text{if} \ n < 0,
  \end{cases}
\]
where $\alpha + \beta = x$ and $\alpha\beta = 1$.

The degree of the $Q$-polynomial of any Kanenobu knot is given in
the following proposition:

\begin{prop}\label{degree}

For the Kanenobu knot $K(p, q)$, we have
\[ \deg Q(p, q) =
  \begin{cases}
  |p|+|q|+6, & \text{if} \ pq \ \geq 0,
     \\
        |p|+|q|+5, & \text{otherwise},
  \end{cases}
\]
\end{prop}
\begin{proof}
We claim that $\sigma_0=S_{-1}$ and  $\sigma_{n} =
\frac{n}{|n|}S_{|n|-1}(x)$  if  $n\neq 0$, where $S_{k}(x)$ is
$k-th$ Chebyshev polynomial of the first kind which is defined
inductively by $S_{-1}(x) = 0, S_{0}(x) = 1$ and $S_{k}(x) =
xS_{k-1}(x) - S_{k-2}(x)$. The  claim is obvious  for $n=0$ and
$n=1$. Now we prove that  $\sigma_{k}$ satisfies  the same inductive
relation as the Chebyshev polynomial.  Since $\alpha+\beta=x$ and
$\alpha \beta=1$, then for all $k\geq 1$, we have:
\begin{align*}
x \sigma_{k} =& (\alpha + \beta)\frac{\alpha^k-\beta^k}{\alpha-\beta}\\
=& \frac{\alpha^{k+1}-\beta^{k+1}+\beta \alpha^k-\alpha \beta^k}{\alpha-\beta}\\
=& \frac{\alpha^k-\beta^k}{\alpha-\beta}+ \frac{\alpha^{k-1}-\beta^{k-1}}{\alpha-\beta}\\
=& \sigma_{k+1}+\sigma_{k-1}.
\end{align*}

Now the claim follows for any $k$ since  $\sigma_{n} = -\sigma_{-n}$
for $n < 0$.

Finally, with the convention  $\deg 0 = -1$, we conclude that  $\deg
Q(p,q) = \deg (\sigma_{|p| +1}\sigma_{|q|+1}) + \deg Q(8_{8}) - 1 =
|p| + |q| + 7 - 1 = |p| + |q| + 6$ for $pq \geq 0$ and $\deg Q(p,q)
= \deg (\sigma_{|p + 1| }\sigma_{|q+1|}) + \deg Q(8_{8}) - 1 = |p| +
|q| + 7 - 1 - 1 = |p| + |q| + 5$ for $pq < 0$.
\end{proof}

\begin{coro}
There are only finitely many Kanenobu knots that are
quasi-alternating.
\end{coro}
\begin{proof}
A necessary condition for the Kanenobu knot to be quasi-alternating,
is as follows:
\[
\deg Q_{K(p,q)} \leq |p| + |q| + 6 < 25.
\]
This implies that $|p| + |q| < 19$ and we know that there are only
finitely many values of $p$ and $q$ that satisfy this inequality.
\end{proof}

The following corollary gives a partial solution of Conjecture 3.10
in \cite{QCQ}.

\begin{coro}\label{important}
The Montesinos link $ L =
M(e;(\alpha_{1},\beta_{1}),(\alpha_{2},\beta_{2}),
\ldots,(\alpha_{r},\beta_{r}),(\alpha,\beta) )$ for all $\alpha = l
+ k\beta$ for $k$ large enough and $ l = 0, 1,\ldots , \beta -1$ in
standard form is not quasi-alternating if $ e = 1$ and
$\sum_{i=1}^{r} \frac{\beta_{i}}{\alpha_{i}} = 1$. This result
supports Conjecture 3.10 in \cite{QCQ}.

\end{coro}

\begin{proof}
According to Theorem \ref{main}, a necessary condition for the above
Montesinos link to be quasi-alternating is as follows:
\[
\deg Q_{L} = c(D) - 2 = c(L) - 2 < \det(L) =
\left(\alpha\prod_{i=1}^{r} \alpha_{i}\right) \left(-1 +
\sum_{i=1}^{r} \frac{\beta_{i}}{\alpha_{i}} +
\frac{\beta}{\alpha}\right)= \beta\prod_{i=1}^{r} \alpha_{i},
\]
the first two equalities follow from \cite[Lemma\,8 \&
Theorem\,10]{LT} with $c(L), c(D)$ are the crossing numbers of the
Montesinos link and its corresponding Montesinos reduced link
diagram respectively. The third equality follows from the formula to
compute the determinant of Montesinos link that first appeared in
\cite[Proposition\,3.1]{CO} and it can be derived from the work in
\cite{QCQ}. Note that increasing the value of $k$ will increase the
value of $c(D)$ while the determinant stays fixed. Therefore, we can
choose $k$ large enough for fixed $\beta$ so that
$\beta\prod_{i=1}^{r} \alpha_{i} \leq c(L) - 2$.

For the second claim, We have to show that
$\frac{\alpha_{i}}{\alpha_{i} - \beta_{i}} \leq \min \{\min \{
\frac{\alpha_{j}}{\beta_{j}}\ | \  j \neq i \},
\frac{\alpha}{\beta}\}$ for any $1 \leq i \leq r$ and
$\frac{\alpha}{\alpha - \beta} \leq \min \{
\frac{\alpha_{j}}{\beta_{j}}\ | \  1 \leq j \leq r \}$. If we choose
$k$ large enough such that $\frac{\alpha}{\beta} = \frac{l}{\beta} +
k > \frac{\alpha_{i}}{\beta_{i}}$ for all $1 \leq i \leq r$, then it
is enough to show $\frac{\alpha_{i}}{\alpha_{i} - \beta_{i}} \leq
\min \{ \frac{\alpha_{j}}{\beta_{j}}\ | \  j \neq i \}$ for any $1
\leq i \leq r$ and $\frac{\alpha}{\alpha - \beta} \leq \min \{
\frac{\alpha_{j}}{\beta_{j}}\ | \  1 \leq j \leq r \}$. For the
first part, suppose $\frac{\alpha_{i}}{\alpha_{i} - \beta_{i}} >
\frac{\alpha_{j}}{\beta_{j}} $ for some $  j$, then we have
$\alpha_{i}\beta_{j} + \alpha_{j}\beta_{i} > \alpha_{j}\alpha_{i}$.
This implies that $\frac{\beta_{j}}{\alpha_{j}} +
\frac{\beta{i}}{\alpha{i}} > 1$ which contradicts the assumption.
Similarly for the second part suppose that $\frac{\alpha}{\alpha -
\beta} > \frac{\alpha_{m}}{\beta_{m}} = \min \{
\frac{\alpha_{j}}{\beta_{j}}\ | \  1 \leq j \leq r \}$ for some $m$,
then we have $\frac{\alpha - \beta}{\alpha} <
\frac{\beta_{m}}{\alpha_{m}}$. Therefore, we obtain
\[
1- \frac{\beta_{m}}{\alpha_{m}} < 1 - \frac{\beta}{\alpha} <
\frac{\beta_{m}}{\alpha_{m}}.
\]
So we conclude that $\frac{1}{2} < \frac{\beta_{m}}{\alpha_{m}}$.
Thus,  $2 < \frac{\alpha}{\beta} = \frac{j+k\beta}{j+(k-1)\beta}.$
This implies that $j + (k-2)\beta < 0 $ which is a contradiction
since $0 \leq j,k-2,\beta$ for large $k$.

\end{proof}

\begin{rem}
Corollary 2.6 explains why most but finitely many of the Montesinos
links of the from $L(m,n) = M(0;(m^{2}+1,m),(n,1),(m^{2} + 1,m)) =
M(1;(m^{2}+1,m),(n,1),(m^{2} + 1,m - m^{2} - 1))$ for positive
integers $m, n$ and large $n$ are not quasi-alternating. This result
was first obtained in \cite[Theorem\,1.3]{G2} for $L(2,3)$ which is
the knot $11n50$. The proof of \cite[Theorem\,1.3]{G2} can be
generalized easily to all $L(m,n)$ with $m > n$ as it was suggested
by J. Greene in \cite[Subsection\,3.2]{G2}.
\end{rem}


In \cite{M}, Manolescu showed that all homologically thin in
Khovanov homology non alternating knots of crossing number less than
or equal to 9 are  quasi-alternating, except  the knot $9_{46}$.
Among the 42 non alternating knots of 10 crossings, 32 are
homologically thin in Khovanov homology. The authors of
\cite{B,CK,G1,M} showed that all these knots are quasi-alternating
except for the knot $10_{140}$. Shumakovitch in \cite{S} showed that
the knots $9_{46}, 10_{140}$ have thick odd Khovanov homology
groups, so they are not quasi-alternating.

Theorem \ref{main} does not characterize quasi-alternating
links since  the knots $9_{46}, 10_{128}$, and 
$11n 50$ for instance, satisfy the  inequality $\deg (Q_L) < det (L)$, but they are not
quasi-alternating.
Actually, the knot $10_{128}$ is homologically thick in Khovanov homology \cite{BM}.
The knot $11n 50$ which  is the Kanenobu knot
$K(3,0)$ does not bound a negative definite 4-manifold with
torsion-free as it has been shown by Greene in
\cite[Theorem\,1.3]{G2}.



\begin{prop}
There are  an infinite family  of knots and an infinite family  of
links which are  not quasi-alternating  but  satisfy the inequality in Theorem
\ref{main}.
\end{prop}

\begin{proof}
The first family is the set of the pretzel knots of the form $P(r +
2, r + 1, -r)$ and the second family is the set of the pretzel links
$P(r + 1, r + 1, -r)$, where  $r > 3$ is an odd integer. It has been shown
in \cite{Q} that these knots and links are
thick in Khovanov homology. Therefore, they are not
quasi-alternating. However, they satisfy the  inequality in Theorem \ref{main}.
\begin{align*}
\deg Q_{P(r + 2, r + 1, -r)} & = 3r + 1\leq r ^{2} - 2 = \det(P(r + 2, r + 1, -r))\\
\deg Q_{P(r + 1, r + 1, -r)} & = 3r + 2 \leq r^{2} - 1 = \det(P(r +
1, r + 1, -r)),
\end{align*}
where the first equality in each of the two equations above follows
from \cite[Lemma\,8 \& Theorem\,10]{LT}.
\end{proof}

\begin{prop}
There is an infinite family of links that are not quasi-alternating
homologically thin in Khovanov homology and satisfy the inequality
in Theorem \ref{main}.
\end{prop}
\begin{proof}
The family is the set of the pretzel links $P(n,n,-n)$ for $n \geq
3$. It has been shown that all these links  are homologically thin
in Khovanov homology in \cite{Q}. However, they  are not
quasi-alternating by Theorem 1.4 in \cite{G2}. It is left to show
that all these links  satisfy the inequality in Theorem \ref{main}.
We have
\[
\deg Q_{P(n,n,-n)} = 3n - 2 \leq n^{2} = \det(P(n,n,-n)),
\]
for $n \geq 3$, where the first equality of above equation follows
from \cite[Lemma\,8 \& Theorem\,10]{LT}.
\end{proof}

Let $L$ be a oriented link. The breadth of the Jones polynomial
$\breadth V_L(t)$ is defined to be the difference between the
highest and the lowest degree of $t$ that appear in $V_L(t)$.
Inspired by Theorem \ref{main} and computations of the breadth and
the determinants of a large number  of links, we conjecture the
following.
\begin{conj}
If $L$ is a quasi-alternating link, then $\breadth V_{L}(t) \leq
det(L)$.
\end{conj}

Conjecture 2.10 is weaker than the one in \cite{QQJ} which states
that for any quasi-alternating link, the crossing number is a lower
bound of the determinant, $c(L)\leq det(L)$ since we know that the
breadth of the Jones polynomial is always less than or equal to the
crossing number of the link. The importance of the latter does also
come from that it solves a conjecture of Greene in \cite{G2} which
states that there are only finitely many quasi-alternating links
with a given determinant. However, Conjecture 2.10 has the advantage
that it involves the breadth of the Jones polynomial which is, in
general, easier to compute than the crossing number. Conjecture 2.10
is true for all quasi-alternating links that have been checked to
satisfy the conjecture $c(L)\leq det(L)$, see \cite{QQJ}. In the
appendix, we prove both conjectures for quasi-alternating closed
3-braids.

We now apply the obstruction criterion introduced by Theorem
\ref{main} to provide a table of knots of 12 crossings or less that
are not quasi-alternating. A second table contains  links of 9
crossings or less that are not quasi-alternating is also provided.
The computations of the Q-polynomials and the determinants are done
using knotinfo \cite{CL}.

Finally, we close this section  with the following  two questions:
\begin{ques}
Can we determine all Kanenobu knots that are quasi-alternating?
\end{ques}

We conjecture that $K(0,0),K(1,0),K(1,-1)$ are the only Kanenobu
knots that are quasi-alternating.

\begin{ques}
Can we characterize all quasi-alternating knots with crossing number
less than or equal to 11?
\end{ques}

Our table combined with the table in \cite{JS} gives a partial
solution for the above question.

\begin{table}[h]
\centering
\begin{tabular}{|c|c|c||c|c|c||c|c|c|}
\hline

Knot & Det. & Deg. & Knot & Det. & Deg. & Knot & Det. & Deg. \\
\hline
$8_{19}$ & 3 & 6 & $9_{42}$ & 7 & 7& $10_{124}$ & 1 & 8\\
\hline
$10_{132}$ & 5 & 8& $10_{139}$ & 3 & 8  & $10_{145}$ & 3& 8 \\
\hline
$10_{153}$ & 1 & 8 & $10_{161}$ & 5 & 6 & $11n 9$ & 5 & 9 \\
\hline
$11n 19$ & 5 & 9& $11n 31$ & 3 & 9& $11n 34$& 1 & 9\\
\hline
$11n 38$ & 3 & 9& $11n 42$ & 1 & 9 & $11n 49$& 1 & 9 \\
\hline
$11n 57$ & 7 & 9 & $11n 67$ & 9 & 9 & $11n 96$ & 7 & 9 \\
\hline
$11n 102$ & 3 & 9 & $11n 104$ & 3 & 9& $11n 111$ & 7 & 9 \\
\hline
$11n 116$ & 1 & 7 & $11n 135$ & 5 & 7 & $11n 139$ & 9 & 9 \\
\hline
$12n 0019$ & 1 & 10 & $12n 0023$ & 9 & 10&  $12n 0031$ & 9 & 10\\
\hline
$12n 0051$ & 9 & 10 & $12n 0056$ & 9 & 9& $12n 0057$ & 9 & 9 \\
\hline
$12n 0096$ & 7 & 10 & $12n 0118$ & 7 & 10& $12n 0121$ & 1 & 10\\
\hline
$12n 0124$ & 7 & 10 & $12n 0129$ & 7 & 10 & $12n 0149$ & 5 & 10\\
\hline
$12n 0175$ & 3 & 10 & $12n 0200$ & 9 & 10& $12n 0210$ & 1 & 10 \\
\hline
$12n 0214$ & 1 & 10 & $12n 0217$ & 5 & 10 & $12n 0221$ & 9 & 9 \\
\hline
$12n 0242$ & 1  & 10 & $12n 0243$ & 5 & 10 & $12n 0268$ & 9 & 10 \\
\hline
$12n 0273$ & 5 & 10 & $12n 0292$ & 1 & 10 & $12n 0293$ & 7 & 10 \\
\hline
$12n 0309$ & 1 & 10 & $12n 0313$ & 1 & 10 & $12n 0318$ & 1 & 8 \\
\hline
$12n 0332$ & 9 & 10 & $12n 0336$ & 5 & 10 & $12n 0352$ & 7 & 10 \\
\hline
$12n 0370$ & 5 & 10 & $12n 0386$ & 9 & 10 & $12n 0402$ & 9 & 10 \\
\hline
$12n 0403$ & 9 & 10 & $12n 0404$ & 3 & 10 & $12n 0419$ & 3 & 10 \\
\hline
$12n 0430$ & 1 & 9 & $12n 0439$ & 3 & 8 & $12n 0446$ & 7 & 10\\
\hline
$12n 0473$ & 1 & 10 & $12n 0475$ & 7 & 10 & $12n 0488$ & 5 & 10 \\
\hline
$12n 0502$ & 9 & 10 & $12n 0519$ & 7 & 8 & $12n 0552$ & 9 & 9 \\
\hline
$12n 0574$ & 9 & 10 & $12n 0575$ & 3 & 10 & $12n 0579 $ & 9 & 9\\
\hline
$12n 0582$ & 9 & 10 & $12n 0591$ & 7 & 8& $12n 0605$ & 9 & 10  \\
\hline
$12n 0617$ & 5 & 9 & $12n 0644$ & 7 & 8 & $12n 0655$ & 3 & 8\\
\hline
 $12n 0673$ & 5 & 10 & $12n 0676$ & 9 & 10 & $12n 0689$ & 7 & 10 \\
\hline
 $12n 0725 $ & 5 & 10 & $12n 0749$ & 7 & 8 & $12n 0812$ & 9 & 9  \\
\hline
$12n 0815$ & 7 & 9 & $12n 0851$ & 5 & 8 & & &\\
\hline
\end{tabular}

\vspace{0.2cm} \caption{Knot table} \label{1}
\end{table}
\begin{table}[h]
\centering
\begin{tabular}{|c|c|c||c|c|c||c|c|c|}
\hline
Link & Det. & Deg. &  Link & Det. & Deg. & Link & Det. & Deg. \\
\hline
$L8n3$ & 4 & 6 & $L8n6$ & 0  & 6 & $L8n8$ & 0 & 5  \\
\hline
$L9n4$ & 4 & 7 & $L9n9$ & 4   & 7 & $L9n12$ & 6 & 7  \\
\hline
$L9n15$ & 2 & 7& $L9n18$ & 2  & 7& $L9n19$ & 0 & 5 \\
\hline
$L9n21$ & 4 & 6 & $L9n27$ & 4  &  7 &  &  &  \\
\hline
\end{tabular}
\vspace{0.2cm} \caption{Link table} \label{3}
\end{table}

\section{Appendix: Proof of Conjecture 2.10 for closed 3-braids}

In this section we prove that Conjecture 2.10 holds for
quasi-alternating closed 3-braids. Although a direct proof can be
given by computing the breadth of the Jones polynomial of closed
3-braids. We prefer here to prove that for any  quasi-alternating
link $L$ of braid index less than or equal to 3, we have $c(L) \leq
det(L)$. This will imply that our conjecture is true for this class
of links.

For $n\geq 2$, let $B_n$ be the  braid group on $n$ strings. It is
well known that $B_n$ is generated by the elementary braids
$\sigma_{1}, \sigma_{2}, \dots ,\sigma_{n-1}$ subject to the
following relations:
\begin{align*}
\sigma_{i} \sigma_{j} & =\sigma_{j} \sigma_{i} \mbox{ if } |i-j| \geq 2 {\mbox{ and }}\\
\sigma_{i} \sigma_{i+1} \sigma_{i} & = \sigma_{i+1}
\sigma_{i}\sigma_{i+1}, \ \forall \ 1 \leq i \leq n-2.
\end{align*}

The importance of braid groups that Alexander's Theorem states that
every link $L$ in $S^3$ can be obtained as the closure of a certain
braid $b$. We write $L=\hat b$. The two generators $\sigma_1$  and
$\sigma_2$ of $B_3$ are pictured in figure \ref{braid}. Closed
3-braids have been classified up to conjugation by Murasugi in
\cite{Mu}.
\begin{figure}
\begin{center}
\includegraphics[width=5cm,height=1.5cm]{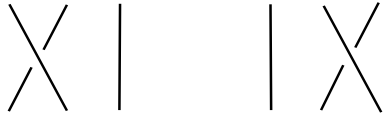}
\caption{The generators $\sigma_1$  and $\sigma_2$ of $B_3$
respectively}
\label{braid}
\end{center}
\end{figure}

\begin{thm}\label{thm}
Let $b$ be a 3-braid and let $h=(\sigma_1\sigma_2)^3$ be a full positive twist. Then $b$ is conjugate to exactly
one of the following:
\begin{enumerate}
\item $h^n\sigma_1^{p_1}\sigma_2^{-q_1}\dots \sigma_1^{p_s}\sigma_2^{-q_s}$, where $s,p_i$ and $q_i$ are positive integers.
\item $h^n\sigma_2^{m}$ where $m \in \Bbb{Z}$.
\item $h^n\sigma_1^{m}\sigma_2^{-1}$, where $m \in \{-1,-2,-3\}$.
\end{enumerate}
\end{thm}
Baldwin in \cite{B} classified quasi-alternating closed 3-braids in
the following theorem:

\begin{thm}\label{thm}
Let $L$ be a closed 3-braid, then
\begin{enumerate}
\item If $L$ is the closure of $h^n\sigma_1^{p_1}\sigma_2^{-q_1}\dots \sigma_1^{p_s}\sigma_2^{-q_s}$,
where $s, p_i$ and $q_i$ are positive integers, then $L$ is quasi-alternating if and only if $n \in \{-1,0,1\}$.
\item If $L$ is the closure of $h^n\sigma_2^{m}$, then $L$ is quasi-alternating if and only if
either $n=1$ and $m \in \{-1,-2,-3\}$ or $n=-1$ and $m \in \{1,2,3\}$.
\item If $L$ is the closure of $h^n\sigma_1^{m}\sigma_2^{-1}$ where $m \in \{-1,-2,-3\}$.
Then $L$ is quasi-alternating  if and only if $n \in \{0,1\}$.
\end{enumerate}
\end{thm}


The following proposition introduces explicit formulas for  the
determinant of any closed 3-braid.

\begin{prop}\label{thm}
\begin{enumerate}
\item Suppose that  $L=\widehat{h^n\sigma_1^{p_1}\sigma_2^{-q_1}\dots \sigma_1^{p_s}\sigma_2^{-q_s}}$,
where $s, p_i$ and $q_i$ are positive integers. Let $p=\sum_{i=1}^{s} p_i$ and $q=\sum_{i=1}^{s} q_i$.
\begin{enumerate}
\item If $n$ is odd, then
$$det(L)= 4 + pq + \displaystyle\sum_{k=2, \\ i_1<\dots<i_k}^{s} p_{i_1}\dots p_{i_k}(q_{i_1}+\dots +q_{i_2-1})\dots
(q_{i_{k-1}}+\dots +q_{i_{k}-1})(q-(q_{i_1}+\dots +q_{i_{k-1}})).$$
\item If $n$ is even, then
$$det(L)=pq+\displaystyle\sum_{k=2, \\ i_1<\dots<i_k}^{s} p_{i_1}\dots p_{i_k}(q_{i_1}+\dots +q_{i_2-1})\dots
(q_{i_{k-1}}+\dots +q_{i_{k}-1})(q-(q_{i_1}+\dots +q_{i_{k-1}}))$$
\end{enumerate}
\item If $L=\widehat{h^n\sigma_2^{m}}$ where $m \in \mathbb{Z}$ then $det(L)=0$ if $n$ is even and $det(L)=4$ if $n$ is odd.\\
\item If $L=\widehat{h^n\sigma_1^{m}\sigma_2^{-1}}$ where $m \in \{-1,-2,-3\}$,
then $det(L)=2$ if $m=-2$  and $det(L)=2+(-1)^{3n+m}$ if $m=-1$ or $-3$.\\
\end{enumerate}

\end{prop}
\begin{proof}
Birman in \cite{Bi} showed that the Jones polynomial of a closed
3-braid $\alpha$ is given by
$$
 V_{\hat\alpha}(t)=(-\sqrt{t})^{e_\alpha}(t+t^{-1}+tr(\psi_t(\alpha))) $$
where  $e_\alpha$ is the exponent sum of $\alpha$ as a word in the
elementary braids $\sigma_1$ and $\sigma_2$. Also, $\psi_t: B_{3}
\longrightarrow GL(2, \mathbb{Z}[t,t^{-1}])$ is the Burau
representation defined by $\psi_t(\sigma_1)=\left [\begin{array}{cc}
-t&1\\
0&1
\end{array} \right ] $ and $\psi_t(\sigma_2)=\left [\begin{array}{cc}
1&0\\
t&-t
\end{array} \right ] $, and $tr$ denotes the usual matrix-trace function.
Recall that for any link $L$, we have $det(L)=|V_L(-1)|$. The values
of the determinants in cases 2 and 3 are obtained easily. Indeed, we
have explicit formulas for the Jones polynomials in these cases:
\begin{align*}
   V_{\widehat{h^n\sigma_2^{m}}}(t) & = (-\sqrt{t})^{m+6n}\left(t+t^{-1}+t^{3n}+(-1)^mt^{3n+m}\right), \\
   V_{\widehat{h^n\sigma_1^{-1}\sigma_2^{-1}}}(t) & =(-\sqrt{t})^{6n-2}\left(t+t^{-1}+t^{3n}(-t)^{-1}\right),\\
   V_{\widehat{h^n\sigma_1^{-2}\sigma_2^{-1}}}(t) & =(-\sqrt{t})^{6n-3}\left(t+t^{-1}\right),\\
   V_{\widehat{h^n\sigma_1^{-3}\sigma_2^{-1}}}(t) & =(-\sqrt{t})^{6n-4}\left(t+t^{-1}+t^{3n}(-t)^{-2}\right).\\
 \end{align*}

Let $\beta=\widehat{\sigma_1^{p_1}\sigma_2^{-q_1}\dots
\sigma_1^{p_s}\sigma_2^{-q_s}}$ and  $\alpha=h^n\beta$. Since
$\psi_t(\sigma_1\sigma_2)^3=t^3.I_2$, then
$V_{\widehat{\alpha}}(t)=(-\sqrt{t})^{e_\alpha}[t+t^{-1}+t^{3n}tr(\psi_t(\beta))]$.
Consequently, $det(\widehat\alpha)=|V_{\widehat
\alpha}(-1)|=|-1-1+(-1)^{n}tr(\psi_{-1}((\beta))|$. On the other
hand, for any positive integers $n$ and $m$, we have:
$\psi_{-1}(\sigma_1^{n})=\left [\begin{array}{cc}
1&n\\
0&1
\end{array} \right ] $ and $\psi_{-1}(\sigma_2^{-m})=\left [\begin{array}{cc}
1&0\\
m&1
\end{array} \right ]. $ Hence  $\psi_{-1}(\sigma_1^n\sigma_2^{-m})=\left [\begin{array}{cc}
1+nm&n\\
m&1
\end{array} \right ].$
A simple induction implies that $tr(\psi_{-1}((\beta))$ is always
positive. In conclusion, $det(\widehat{\alpha})=det(\widehat\beta)$
if $n$ is even and $det(\widehat{\alpha})=det(\widehat\beta)+4$, if
$n$ is odd.

To prove the formula in the  case 1, we first compute  the
determinant of the alternating link $\widehat
\beta=\widehat{\sigma_1^{p_1}\sigma_2^{-q_1}\dots
\sigma_1^{p_s}\sigma_2^{-q_s}}$  by counting the number of spanning
trees of the Tutte graph associated with the diagram of
$\widehat\beta$ given in figure \ref{tuttegraph}. This graph is made
up of a cycle $u_1u_2\dots u_qu_1$ together with  an extra vertex
$w$ of degree $p$ which is connected to every vertex $u_{q_i}$ by
$p_i$ parallel edges.

\begin{figure}
\begin{center}
\includegraphics[width=6cm,height=6cm]{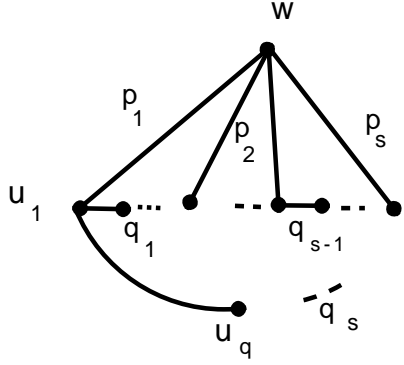}
\end{center}
\caption{The Tutte graph associated with the diagram of
$\widehat\beta$}
\label{tuttegraph}
\end{figure}

A spanning tree where $w$ has degree 1, consists  of  an edge among
the $p$ edges incident to $w$ and a spanning tree of the cycle.
There are obviously $pq$ such trees. For a spanning  tree where $w$
has degree $k\geq 2$, we first  make a choice of $k$ mutually non
parallel edges incident to $w$, say $wu_{q_{i_1}},\dots,
wu_{q_{i_k}} $. Then we break all the cycles of type $wu_
{q_{i_r}}\dots wu_ {q_{i_{r+1}}}w$ and the cycle $wu_{q_{i_k}}\dots
u_{q_{i_{q}}}w$   by removing from each cycle an edge not incident
to $w$. Note that  the number of spanning trees will be then the
product $p_{i_1}\dots p_{i_k}(q_{i_1}+\dots+ q_{i_2-1}) \dots
(q_{i_{k-1}}+\dots +q_{i_{k}-1})(q-(q_{i_1}\dots q_{i_k-1}))$. The
total number of spanning trees is obtained by taking the sum through
all $k \leq s$ and $i_1 < \dots < i_k$. If $n$ is even, then
$det(L)=det(\widehat \beta)$. However, if $n$ is odd, then
$det(L)=|-1-1-tr(\psi_{-1}(\beta)|=tr(\psi_{-1}(\beta))+2=det(\widehat\beta)+4$.
\end{proof}

Now, we  prove that for all quasi-alternating closed 3-braids, we
have $c(L)\leq det(L)$. We start by considering the class of links
in the first case of Baldwin's Theorem. If $n=0$,  then the link is
alternating and the result holds. We will prove the result  for
$n=1$, the case $n=-1$ is similar. Assume that  $L=\widehat{
(\sigma_1\sigma_2)^3\sigma_1^{p_1}\sigma_2^{-q_1}\dots
\sigma_1^{p_s}\sigma_2^{-q_s}}$, then: If $s>1$, or ($s=1, p_1>1$
and $q_1>1$),  then by the proposition above $det(L)\geq 4+pq \geq
4+p+q$. On the other hand we have

$$\begin{array}{rl}
(\sigma_1\sigma_2)^3\sigma_1^{p_1}\sigma_2^{-q_1}\dots \sigma_1^{p_s}\sigma_2^{-q_s}&\equiv
(\sigma_2\sigma_1)^3\sigma_1^{p_1}\sigma_2^{-q_1}\dots \sigma_1^{p_s}\sigma_2^{-q_s}\\
& \equiv \sigma_2\sigma_1\sigma_2\sigma_1\sigma_2\sigma_1\sigma_1^{p_1}\sigma_2^{-q_1}\dots \sigma_1^{p_s}\sigma_2^{-q_s}\\
& \equiv \sigma_1\sigma_2\sigma_1\sigma_2\sigma_1\sigma_1^{p_1}\sigma_2^{-q_1}\dots \sigma_1^{p_s}\sigma_2^{-q_s+1}.\\
\end{array}$$

Thus, $c(L)\leq 5+p+q-1=4+p+q \leq det(L)$.

If $s=1$, $p=1$ and $q>1$, then $det(L)=4+q$, and

$$\begin{array}{rl}
(\sigma_1\sigma_2)^3\sigma_1\sigma_2^{-q}&\equiv
\sigma_1\sigma_2\sigma_1\sigma_2 \sigma_1\sigma_2 \sigma_1
\sigma_2^{-q+1}\sigma_2^{-1}\\&\equiv
\sigma_2\sigma_1\sigma_2\sigma_1
\sigma_2\sigma_1 \sigma_1 \sigma_2^{-q+1}\sigma_2^{-1}\\
& \equiv \sigma_1\sigma_2\sigma_1\sigma_2\sigma_1 \sigma_1 \sigma_2^{-q+1} \\
& \equiv \sigma_2\sigma_1\sigma_2\sigma_2\sigma_1 \sigma_1 \sigma_2^{-q+1}\\
& \equiv \sigma_1\sigma_2\sigma_2\sigma_1 \sigma_1 \sigma_2^{-q+2}.\\
\end{array}$$

Thus, $c(L)\leq 3+q \leq det(L)$.

If $s=1$, $p>1$ and $q=1$, then $det(L)=4+p$, and

$$\begin{array}{rl}
(\sigma_1\sigma_2)^3\sigma_1^{p}\sigma_2^{-1}&\equiv
\sigma_1\sigma_2 \sigma_1\sigma_2\sigma_1\sigma_2 \sigma_1^p
\sigma_2^{-1}\\&\equiv \sigma_2\sigma_1
\sigma_2\sigma_1\sigma_2\sigma_1 \sigma_1^p \sigma_2^{-1}\\
& \equiv \sigma_1\sigma_2\sigma_1\sigma_2\sigma_1 \sigma_1^{p} \\
& \equiv \sigma_1\sigma_1\sigma_2\sigma_1\sigma_1^{p+1}\\
& \equiv \sigma_1^{p+4}.\\
\end{array}$$

Thus, $c(L)=p+4=det(L).$

If $s=1$, $p=1$ and $q=1$, then  $det(L)=5$ and
$$\begin{array}{rl}
(\sigma_1\sigma_2)^3\sigma_1\sigma_2^{-1}&\equiv \sigma_1\sigma_2
\sigma_1\sigma_2\sigma_1\sigma_2 \sigma_1 \sigma_2^{-1}\\&\equiv
\sigma_2\sigma_1
\sigma_2\sigma_1\sigma_2\sigma_1 \sigma_1 \sigma_2^{-1}\\
& \equiv \sigma_1\sigma_2\sigma_1\sigma_2\sigma_1^{2} \\
& \equiv \sigma_1\sigma_1\sigma_2\sigma_2\sigma_1^{2}\\
& \equiv \sigma_1^{2}\sigma_2\sigma_1^{3}\\
&\equiv \sigma_1^{5}.
\end{array}$$

Thus, $c(L)=5=det(L)$.

The two other cases in Baldwin's Theorem involve only a finite
number of knots. A routine case by case check shows that the result
holds for all these knots. This ends the proof of Conjecture 2.10 in
the case of closed 3-braids.

\end{document}